\documentclass[leqno,10pt,a4paper,intlimits]{amsart}
\usepackage{graphicx,color}

\usepackage{amsmath,amssymb,amsthm}
\usepackage{enumerate}
\usepackage{mathrsfs}
\usepackage{hhline}

\def\calC{\mathcal{C}}
\def\calW{\mathcal{W}}
\def\calX{\mathcal{X}}

\def\calA{\mathcal{A}}
\def\calB{\mathcal{B}}

\def\calD{\mathcal{D}}
\def\calT{\mathcal{T}}
\def\calS{\mathcal{S}}

\def\calR{\mathcal{R}}
\def\calG{\mathcal{G}}
\def\calP{\mathcal{P}}
\def\calJ{\mathcal{J}}

\def\C{\mathrm{C}}
\def\Ell{\mathrm{L}}
\def\LLL{\mathscr{L}}
\def\BUC{\mathrm{BUC}\hspace*{0.5mm}}

\def\BB{\mathrm{B}\hspace*{0.5mm}}
\def\Sob{\mathrm{H}\hspace*{0.5mm}}

\def\reqtext#1{#1}
\def\bdef{\textbf}

\def\dd{\mathrm{d}}
\def\dt{\tfrac{\dd}{\dd t}}
\def\Id{I}
\def\Dt{\tau}
\def\Dx{\delta}

\def\ee{\mathrm{e}}

\def\NN{\mathbb{N}}
\def\RR{{\mathbb{R}}}
\def\CC{{\mathbb{C}}}

\def\dom{D}

\newtheorem{theorem}{Theorem}[section]
\newtheorem{corollary}[theorem]{Corollary}
\newtheorem{proposition}[theorem]{Proposition}
\theoremstyle{definition}
\newtheorem{definition}[theorem]{Definition}
\newtheorem{assumptions}[theorem]{Assumption}
\newtheorem{remark}[theorem]{Remark}

\begin{document}

\title[Non-autonomous splitting]{Operator splitting for non-autonomous evolution equations}

\author[A. B\'{a}tkai]{Andr\'{a}s B\'{a}tkai}
\address{ELTE TTK, Institute of Mathematics\newline 1117 Budapest, P\'{a}zm\'{a}ny P. s\'{e}t\'{a}ny 1/C, Hungary.}
\email{batka@cs.elte.hu}
\thanks{Research partially supported by the Alexander von Humboldt-Stiftung.}

\author[P. Csom\'{o}s]{Petra Csom\'{o}s}
\address{Technische Universit\"{a}t Darmstadt, Fachbereich Mathematik\newline Schlo{\ss}gartenstr. 7, 64289 Darmstadt, Germany}
\email{csomos@mathematik.tu-darmstadt.de}
\email{farkas@mathematik.tu-darmstadt.de}
\author[B. Farkas]{B\'{a}lint Farkas}

\author[G. Nickel]{Gregor Nickel}
\address{Universit\"at Siegen, FB 6 Mathematik\newline Walter-Flex-Str. 3, 57068 Siegen, Germany.}
\email{nickel@mathematik.uni-siegen.de}


\date\today

\dedicatory{To Ulf Schlotterbeck, our inspirator, on his 70th birthday.}

\begin{abstract}
 We establish general product formulas for the solutions of non-autonomous abstract Cauchy problems.
 The main technical tools are evolution semigroups allowing the direct application of existing results on autonomous problems.
 The results obtained are illustrated by the example of an autonomous diffusion equation perturbed with time dependent potential.
 We also prove convergence rates for the sequential splitting applied to this problem.
\end{abstract}
\maketitle


\section{Introduction}

Operator splitting procedures are used to solve ordinary and partial
differential equations numerically. They can be considered as
certain finite difference methods which simplify or even make the
numerical treatment of differential equations possible. The idea
behind these procedures is the following. In many situations, a
certain physical phenomenon can be considered as the combined effect
of several processes. Hence the behavior of a physical quantity is
described by a partial differential equation in which the time
derivative depends on the sum of operators corresponding to the
different processes. These operators usually are of different nature
and for each sub-problem corresponding to each operator there might
be an effective numerical method providing fast and accurate
solutions. For the sum of these operators, however, it is not always
possible to find an adequate and effective method. Hence, the idea
of operator splitting procedures means that instead of the sum we
treat the operators separately and the solution of the original
problem is then to be recovered from the numerical solutions of
these sub-problems. We refer to the recent monographs by Farag\'o and
Havasi \cite{Farago-Havasi_book} or Holden et al. \cite{Holden-Karlsen-Lie-Risebro} for a detailed introduction to the
theory and applications of operator splitting methods.

There was enormous progress in recent years in the theoretical
investigation of operator splitting procedures. Especially, ordinary
differential equations and autonomous linear evolution equations
have been treated thoroughly, see also B\'{a}tkai, Csom\'{o}s and Nickel
\cite{Batkai-Csomos-Nickel} and the subsection below for a
(certainly not complete) list of references.

The aim of the present paper is to investigate the above described
splitting method for non-autonomous evolution equations of the form
\begin{equation}\label{eq:perturbed}\tag{NCP}
\begin{cases}
\dt{u}(t)=(A(t)+B(t))u(t), \quad t\geq s\in\RR, \\
u(s)=x\in X,
\end{cases}
\end{equation}
on some Banach space $X$.
Our particular goal is to emphasize that \emph{non-autonomous}
evolution equations can often be rewritten as an \emph{autonomous}
abstract Cauchy problem by means of an appropriate choice for the
state-space. Thus, by making use of so-called evolution semigroups,
it is possible to apply existing results for autonomous problems.

First we summarize the necessary background on operator splitting
for abstract Cauchy problems, i.e., operator splitting in the
framework of strongly continuous operator semigroups. The key
ingredient here is Chernoff's Theorem \ref{thm:chernoff}. Then
non-autonomous evolution equations and evolution semigroups are
surveyed, providing the main technical tools for the succeeding
sections.

A product representation is presented in Section \ref{sec:4}, while
operator splitting --- strictly in the sense above --- is considered
in Section \ref{sec:5}. To keep our presentation short, we mainly
restrict ourselves to the case of the so-called \emph{sequential
splitting}, but in Section \ref{sec:6} we show how higher order
splitting methods can be treated with essentially no difference. In
that section, we also prove the convergence of the splitting methods
when combined with spatial ``discretization,'' and make a quick outlook on the positivity  of evolution families. Finally, as an
illustration of the developed tools, we apply them to a diffusion
equation with time dependent potential. Moreover, by semigroup
methods, using results of Jahnke and Lubich \cite{Jahnke-Lubich}, and
Hansen and Ostermann \cite{Hansen-Ostermann,Hansen-Ostermann2}, we obtain estimates on
the order of the convergence.

A word on notation: For a family of operators
$U_0,U_1,\ldots,U_{n-1}\in\LLL(X)$, we denote the (``time-ordered'')
product of these operators by
\begin{equation*}
\prod_{p=0}^{n-1}U_p:=U_{n-1} U_{n-2} \cdots U_1
U_0\quad\mbox{and}\quad\prod_{p=n-1}^{0}U_p:=U_0U_1\cdots U_{n-2}
U_{n-1} .
\end{equation*}

\subsection*{Operator splitting for autonomous
problems}

In this section, we recollect the main notions and results of
operator splitting for \emph{autonomous} equations. Consider the
following abstract Cauchy problem on a given Banach space $X$:
\begin{equation*}
\begin{cases}
\dt u(t)&=(A+B)u(t), \qquad t\ge 0, \\
u(0)&=x\in X, \label{acpspl}
\end{cases}
\tag{$\mathrm{ACP}$}
\end{equation*}
where the operators $A$, $B$, and the closure $C:=\overline{A+B}$
are supposed to be generators of strongly continuous semigroups $T$,
$S$, and $U$, respectively. Our general reference on strongly
continuous operator semigroups is the monograph Engel and Nagel
\cite{Engel-Nagel}.

As mentioned in the introduction, operator splitting means that we
try to recover the solution semigroup $U$ using the semigroups $T$
and $S$. As for splitting procedures we mention the most frequently
used ones (for more details, see B\'{a}tkai, Csom\'{o}s and Nickel
\cite[Section 2.2]{Batkai-Csomos-Nickel}):

\medskip
\noindent\textbullet\hangafter1\hangindent1em\hskip0.5emThe
\bdef{sequential splitting}, classically the Lie-Trotter product
formula, is given by
\begin{align*}
u^\text{sq}_n(t)&:=[S(t/n)T(t/n)]^nx,
\intertext{\hangafter1\hangindent1em \textbullet\hskip0.5emthe
\bdef{Strang} splitting is given by}
u^\text{St}_n(t)&:=[T(t/2n)S(t/n)T(t/2n)]^nx,
\intertext{\noindent\hangafter1\hangindent1em \textbullet\hskip0.5em
and --- for a fixed parameter $\Theta\in(0,1)$ --- the
\bdef{weighted splitting} is} u^\text{w}_n(t)&:=[\Theta
S(t/n)T(t/n)+(1-\Theta)T(t/n)S(t/n)]^n x
\end{align*}
\par\noindent with $n\in\mathbb{N}$. In case $\Theta=\tfrac{1}{2}$, it is also called symmetrically weighted splitting. The convergence of these procedures is usually ensured by the
following classical result.
\begin{theorem}[Chernoff \cite{Chernoff}, or see Engel and Nagel {\cite[Sec.~III.]{Engel-Nagel}}]\label{thm:chernoff}
Let $C$ be a linear operator in the Banach space $X$ and assume that
$F:\RR_+\to \LLL(X)$ is a (strongly) continuous function with
$F(0)=I$ and
\begin{equation*}
\|(F(t))^k\|\leq M\ee^{k\omega t} \quad\mbox{for all $t\geq 0$
 and  $k\in \NN$} \quad \reqtext{\bdef{(stability)}}.
\end{equation*}
Suppose that there is a dense subspace $D$, with $(\lambda-C)D$
being also dense for some (large) $\lambda>0$. If for every $x\in D$
the limit
\begin{equation*}
\lim_{h\to 0}\frac{F(h)x-x}{h}=Cx\quad
\reqtext{\bdef{(consistency)}}
\end{equation*}
exists, then $C$ is the generator of a $C_0$-semigroup $U$, the set $D$ is a core for the generator $C$, and we
have
\begin{equation*}
\lim_{n\to
\infty}\bigl(F(\tfrac{t}{n})\bigr)^nx=U(t)x\quad\reqtext{\bdef{(convergence)}}.
\end{equation*}
\end{theorem}
Note that if the closure of $C$ is already known to be a generator, as it is the case in problems motivated by numerical analysis, then the range condition is automatically satisfied.

The operator family $F$ is sometimes called a \bdef{finite
difference method}. Clearly, the above mentioned splitting
procedures have this form. For example, for the sequential splitting
we take
\begin{equation*}
F^\text{sq}(h)=S(h)T(h).
\end{equation*}%

\medskip

It is important to note that  Chernoff's Theorem does not yield anything a priori about
the rate of convergence. The finite difference method $F$ is said to
be of \bdef{order $p>0$}, if for $x$ from a suitably large subset of
$X$ there is $C>0$ such that for all $t\in [0,t_0]$ we have
\begin{align*}
\bigl\| F(\tfrac{t}{n})^n x - U(t)x \bigr\| &\leq \tfrac{C}{n^p},
\intertext{or, as in many special cases, equivalently,} \|F(h)x -
U(h)x\| &\leq C'h^{p+1}.
\end{align*}

The equivalence holds in special cases where it is possible to
ensure the invariance of the above mentioned large subset $D$ of $X$
(for more details we refer to the Lax equivalence theorem which
states that the above two definitions are equivalent for a finite
different method if and only if the method is stable).

Different splitting procedures were introduced to increase the order
of convergence. In the finite dimensional setting, it is well known
that the sequential splitting is of first order, the Strang and the
weighted splitting with $\Theta=\tfrac{1}{2}$ are of second order.
Moreover, the weighted splitting allows also the use of parallel
computing.

In the infinite dimensional case, however, no similar general
statement can be made without additional assumptions. There has been
intense research in this direction, and we mention the works by
Bj{\o}rhus \cite{Bjorhus}, Cachia and Zagrebnov \cite{Cachia-Zagrebnov}, Farag\'o and Havasi \cite{Farago-Havasi}, Hansen and Ostermann
\cite{Hansen-Ostermann}, Ichinose et al. \cite{Ichinose-Neidhardt-Zagrebnov}, Jahnke and Lubich \cite{Jahnke-Lubich} or Neidhardt and Zagrebnov \cite{Neidhardt-Zagrebnov1}.

To obtain error estimates later for diffusion problems, we apply a
result by Jahnke-Lubich, Hansen-Ostermann, which  relies on
commutator bounds. For simplicity, we mention here only the special
case  used later.

\begin{theorem}[Jahnke and Lubich {\cite[Theorem 2.1]{Jahnke-Lubich}}, Hansen and Ostermann {\cite[Theorem 2.3]{Hansen-Ostermann}}]\label{Thm:Jahnke-Lubich}
Suppose that $A$ generates a strongly continuous contraction
semigroup $\ee^{t A}$ in the Banach space $X$ and that
$B\in\LLL(X)$ such that there exists an $\alpha>0$ such that
\begin{equation}\label{eq:commutator_cond}
\|[A,B]v\|=\|(AB-BA)v\| \leq c\bigl\|(-A)^{\alpha}v\bigr\|
\end{equation}
for all $v\in D\subseteq \dom((-A)^{\alpha})$ (where $D$ is some
dense subspace of $\dom((-A)^{\alpha})$ invariant under $\ee^{t(A+B)}$). Then one has first order
convergence for the sequential and Strang splittings, i.e.,
\begin{align*}
\Bigl\| \left(\ee^{\frac{t}{n}B}\ee^{\frac{t}{n}A}\right)^n v - \ee^{t(A+B)}v \Bigr\| &\leq \frac{Ct^2}{n} \Bigl\|(-A)^{\alpha}v\Bigr\|,\\
\Bigl\|
\left(\ee^{\frac{t}{2n}A}\ee^{\frac{t}{n}B}\ee^{\frac{t}{2n}A}\right)^n
v - \ee^{t(A+B)}v \Bigr\| &\leq \frac{Ct^2}{n}
\Bigl\|(-A)^{\alpha}v\Bigr\|.
\end{align*}
\end{theorem}


\subsection*{Non-autonomous evolution equations and evolution
semigroups}

In this section we summarize the main results and definitions on
\textit{non-autonomous} evolution equations and evolution semigroups needed
for our later exposition. For a detailed account and bibliographic
references see, e.g., the survey by Schnaubelt in \cite[Section
VI.9.]{Engel-Nagel}. Consider now the non-autonomous evolution
equation
\begin{equation*}
\begin{cases}\tag{$\mathrm{NCP}_{s,x}$}
\dt u(t)=A(t)u(t), \quad t\geq s\in\RR, \\
u(s)=x\in X,
\end{cases}
\end{equation*}
where $X$ is a Banach space, $\big(A(t), \dom(A(t))\big)$ is a
family of (usually unbounded) linear operators on $X$.

\begin{definition} A continuous function $u : [s,\infty) \rightarrow X$
is called a \bdef{(classical) solution} of ($\mathrm{NCP}_{s,x}$) if
$u \in \C^1([s, \infty);X)$, $u(t) \in \dom(A(t))$ for all $t \ge
s$, $u(s) = x$, and $\dt{u}(t) = A(t) u(t)$ for $t \ge s$.
\end{definition}

We use the following slight modification of Kellermann's definition
\cite[Definition 1.1]{Kellermann} for the well-posedness of the
non-autonomous Cauchy problem $\mathrm{(NCP)}$.
\begin{definition}[{\bdef{Well-posedness}}\rm] \label{Well-NCP}
For a family $\big(A(t),\dom(A(t))\big)_{t \in \RR}$ of linear
operators on the Banach space $X$ the non-autonomous Cauchy problem
$\mathrm{(NCP)}$ is called \bdef{well-posed} (with regularity
subspaces $(Y_s)_{s \in \RR}$ and exponentially bounded solutions)
if the following are true.
\begin{enumerate}[(i)]
\item \bdef{(Existence)} For all $s \in \RR$ the subspace
\begin{equation*} Y_s := \Bigl\{ y \in X \; : \; \mbox{ there exists a classical solution for } \text{(NCP)}_{s,y}\Bigr\} \subset \dom(A(s))
\end{equation*}
is dense in $X$.
\item \bdef{(Uniqueness)} For every $y \in Y_s$ the solution $u_s(\cdot,y)$ is unique.
\item \bdef{(Continuous dependence)} The solution depends continuously on $s$ and $y$, i.e.,  if $s_n \to s \in \RR$, $y_n \to y \in Y_s$ with $y_n \in Y_{s_n}$, then we have
\begin{equation*}
\| \hat{u}_{s_n}(t,y_n) - \hat{u}_s(t,y) \| \to 0
\end{equation*}
uniformly for $t$ in compact subsets of $\RR$, where
\begin{equation*}
\hat{u}_r(t,y) :=
\begin{cases}
 u_r(t,y) &\mbox{if } r \leq t, \\
y &\mbox{if } r > t.
\end{cases}
\end{equation*}
\item \bdef{(Exponential boundedness)} There exist constants $M \ge 1$ and $\omega \in \RR$
such that
\begin{equation*}
\| u_s(t, y)\| \leq M \ee^{\omega (t-s)} \| y \|
\end{equation*}
for all $y \in Y_s$ and $t\ge s$.
\end{enumerate}
\end{definition}

\noindent As in the autonomous case, the operator family solving a
non-autonomous Cauchy problem enjoys certain algebraic properties.

\begin{definition}[\bdef{Evolution family}] \label{evf}
A family $U=(U(t,s))_{t \ge s}$ of linear, bounded operators on a
Banach space $X$ is called an (exponentially bounded)
\bdef{evolution family} if
\begin{enumerate}
\item[(i)] $U(t,r)U(r,s) = U(t,s), \quad U(t,t) = \Id$ \quad holds for all
$t \ge r \ge s \in \RR$,
\item[(ii)] the mapping $(t,s) \mapsto U(t,s)$ is strongly continuous,
\item[(iii)] $\| U(t,s) \| \leq M \ee^{\omega (t-s)}$
for some $M \ge 1,
 \omega \in \RR$ and all $t \ge s \in \RR$.
\end{enumerate}
\end{definition}

In general, however, and in contrast to the behavior of
$C_0$-semigroups (i.e., the autonomous case), the algebraic
properties of an evolution family do not imply any differentiability
on a dense subspace. So we have to add some differentiability
assumptions in order to solve a non-autonomous Cauchy problem by an
evolution family.
\begin{definition}\label{Well-evf}
An evolution family $U=(U(t,s))_{t \ge s}$ is called \bdef{evolution
family solving $\mathrm{(NCP)}$} if for every $s \in \RR$ the
regularity subspace
\begin{equation*}
Y_s := \Bigl\{ y \in X \; : \; [s, \infty) \ni t \mapsto U(t,s)y
\mbox{ solves }  \text{(NCP)}_{s,y} \Bigr\}
\end{equation*}
is dense in $X$.
\end{definition}

The well-posedness of $\mathrm{(NCP)}$ can now be characterized by
the existence of a solving evolution family.

\begin{proposition}[{Nickel \cite[Proposition 2.5]{nickel97}}]
Let $X$ be a Banach space, and assume that $\big(A(t),
\dom(A(t))\big)_{t \in \RR}$ is a family of linear operators on $X$
and consider the non-autonomous Cauchy problem $\mathrm{(NCP)}$. The
following assertions are equivalent.
\begin{enumerate}
\item[(i)]  The non-autonomous Cauchy problem $\mathrm{(NCP)}$ is well-posed.
\item[(ii)] There exists a unique evolution family
  $(U(t,s))_{t \ge s}$ solving $\mathrm{(NCP)}$.
\end{enumerate}
\end{proposition}

To every evolution family we can associate  $C_0$-semigroups on
$X$-valued function spaces. These semigroups, which determine the behavior
of the evolution family completely, are called \emph{evolution
semigroups}. Consider the Banach space
\begin{equation*}
\BUC(\RR; X) =\bigl\{ f : \RR \to X \; : \;f \mbox{ is bounded and uniformly
continuous}\bigr\},
\end{equation*}
normed by
\begin{equation*}
\| f \| := \sup_{t \in \RR} \| f(t) \|, \quad f \in\BUC(\RR;X);
\end{equation*}
or any closed subspace of it that is invariant under the right
translation semigroup $\calR$ defined by
\begin{equation*}
(\calR(t)f)(s) := f(s-t) \quad \mbox{ for } f \in\BUC(\RR;X) \mbox{
and } s \in \RR, \; t \ge 0.
\end{equation*}
In the following $\calX$ will denote such a closed subspace; we
shall typically take $\calX=\C_0(\RR;X)$, the space of continuous
functions vanishing at infinity.

It is easy to check that the following definition yields a strongly
continuous semigroup.
\begin{definition} \label{EHG}
For an evolution family $U=(U(t,s))_{t \ge s}$ we define the
corresponding \bdef{evolution semigroup} $\calT$ on the space
$\calX$ by
\begin{equation*}
 (\calT(t)f)(s) := U(s,s-t)f(s-t)
\end{equation*}
for $f \in\calX$, $s \in \RR$ and $t \ge 0$. We denote its
infinitesimal generator by $(\calG,\dom(\calG))$.
\end{definition}
With the above notation, the evolution semigroup operators can be
written as
\begin{equation*}
\calT(t)f = U(\cdot, \cdot -t) \calR(t)f.
\end{equation*}
We can recover the evolution family from the evolution semigroup by
choosing a function $f \in\calX$ with $f(s) =x$. Then we obtain
\begin{equation}\label{eq:recovery}
U(t,s)x = (\calR(s-t) \calT(t-s)f)(s)
\end{equation}
for every $s \in \RR$ and $t \ge s$.

The generator of the right translation semigroup is essentially the
differentiation $-\frac{\dd}{\dd s}$ with domain
\begin{equation*}
\dom(-\tfrac{\dd}{\dd s}):=\calX_1:= \bigl\{ f \in \C^1(\RR; X) : f,
f' \in\calX \bigr\}.
\end{equation*}
For a family $\big(A(t), \dom(A(t))\big)_{t \in \RR}$ of unbounded
operators on $X$ we consider the corresponding multiplication
operator $\big(A(\cdot), \dom(A(\cdot))\big)$ on the space $\calX$
with domain
\begin{equation*}
\dom(A(\cdot)):=\bigl\{f \in\calX  : \: f(s) \in
\dom(A(s))\:\forall\:s \in \RR, \mbox{ and } [ s \mapsto A(s)f(s)]
\in\calX \bigr\},
\end{equation*}
and defined by
\begin{equation*}
(A(\cdot)f)(s):=A(s)f(s) \mbox{ for all } s \in \RR.
\end{equation*}

Now we characterize well-posedness for non-autonomous Cauchy
problems.
\begin{theorem}[{Nickel \cite[Theorem 2.9]{nickel97}}] \label{WellChar}
Given a Banach space $X$, and a family of linear operators
$\big(A(t), \dom(A(t))\big)_{t \in \RR}$  on $X$. The following
assertions are equivalent.
\begin{enumerate}[(i)]
\item The non-autonomous Cauchy problem $\mathrm{(NCP)}$ for the family $(A(t))_{t \in \RR}$ is well-posed (with exponentially bounded solutions).
\item There exists a unique evolution semigroup $\calT$ with
generator $(\calG, \dom(\calG))$ and an invariant core $\calD
\subseteq \calX_1 \cap \dom(\calG)$ such that
\begin{equation*} \calG f + f' = A(\cdot)f
\end{equation*}
for all $f \in \calD$.
\end{enumerate}
\end{theorem}%

Conditions implying well-posedness are generally divided into
assumptions of \emph{``parabolic''} and of \emph{``hyperbolic''}
type. Roughly speaking, the main difference between these two types
is that in the parabolic case we assume all $A(t)$ being generators
of analytic semigroups, while in the hyperbolic case we assume the
stability for certain products instead. In both cases one has to add
some continuity assumption on the mapping $t \mapsto A(t)$.
We mention only a typical and quite simple version for each type.

\begin{assumptions}[\bdef{Parabolic case}] \label{asu-para} \label{A(inh-para)}
\rule{0pt}{0pt}
\begin{enumerate}[(P1)]
\item \label{P1} The domain $D:= \dom(A(t))$ is dense in $X$
and is independent of $t \in \RR$.
\item\label{P2} For each $t \in \RR$ the operator $A(t)$ is
the generator of an analytic semigroup $\ee^{\cdot A(t)}$. For all
$t \in \RR$, the resolvent $R(\lambda, A(t))$ exists for all
$\lambda \in \CC$ with $\Re \lambda \ge 0$ and there is a constant
$M \ge1$ such that
\begin{equation*} \| R(\lambda, A(t)) \| \leq \frac{M}{|\lambda| +1}
\end{equation*}
for $\Re \, \lambda \ge 0$, $t \in \RR$. The semigroups $\ee^{\cdot
A(t)}$ satisfy $\|\ee^{s A(t)}\| \leq M\ee^{\omega s}$
for absolute constants $\omega < 0$ and $M \ge 1$.
\item \label{P3} There exist constants $L \ge 0$ and
$0 < \alpha \leq 1$ such that
\begin{equation*}
\| (A(t) - A(s))A(0)^{-1} \| \leq L |t-s|^{\alpha} \mbox{ for all }
t,s \in \RR.
\end{equation*}
\end{enumerate}
\end{assumptions}
\begin{assumptions}[\bdef{Hyperbolic case}] \label{asu-hyp}
\rule{0pt}{0pt}
\begin{enumerate}[(H1)]
\item The family $(A(t))_{t \in \RR}$ is \bdef{stable},
i.e., all operators $A(t)$ are generators of $C_0$-semi\-groups and
there exist constants $M \ge 1$ and $\omega \in \RR$ such that
\begin{equation*} (\omega, \infty) \subset \rho(A(t)) \quad \mbox{for all } t \in \RR
\end{equation*}
and
\begin{equation*} \label{stab_hyp}
\Bigl\| \prod_{j=1}^{k} R(\lambda, A(t_j)) \Bigr\| \leq M (\lambda
-\omega)^{-k} \quad \mbox{for all } \lambda > \omega
\end{equation*}
and every finite sequence $- \infty < t_1 \leq t_2 \leq \dots \leq
t_k < \infty$, $k\in \NN$.
\item There exists a densely embedded subspace
$Y \hookrightarrow X$, which is a core for every $A(t)$ such that
the family of the parts $(A_{|Y}(t))_{t \in \RR}$ in $Y$ is a stable
family on the space $Y$.
\item The mapping
$\RR \ni t \mapsto A(t) \in {\mathcal L}(Y,X)$ is uniformly
continuous.
\end{enumerate}
\end{assumptions}

\begin{remark}
	Since the classical papers of Evans \cite{Evans}, Howland \cite{Howland}, and Neidhardt \cite{Neidhardt1,Neidhardt2,Neidhardt3},
    evolution semigroups have been intensively used to study non-autonomous evolution equations. Here, various results on well-posedness
	as well as qualitative behavior of these equations were obtained. For a quite comprehensive overview and a long list
	of different variants we refer, e.g., to Nagel and Nickel \cite{Nagel-Nickel}, Neidhardt and Zagrebnov \cite{Neidhardt-Zagrebnov},
	Nickel \cite{nickel97, nickel00}, Nickel and Schnaubelt \cite{Nickel-Schnaubelt} and Schnaubelt \cite{schnaubelt02}.
	The recent article Neidhardt and Zagrebnov \cite{Neidhardt-Zagrebnov} focuses on (quite general) assumptions of the {``hyperbolic''} type and obtains well-posedness results -- in a general sense -- for non-autonomous evolution equations by properly defining and analyzing the ``sum'' $-\frac{\dd}{\dd s} + A(\cdot)$ yielding the generator of the associated evolution semigroup.
	In contrast to that approach, in our paper we simply assume well-posedness of our evolution equation under any appropriate
	(parabolic or hyperbolic) condition. Therefore, the solving evolution family, the corresponding evolution semigroup,
	and its generator are well defined by assumption. Our main interest is then, how these solutions can be approximated
	(numerically) by splitting procedures.
\end{remark}


\section{A product formula}\label{sec:4}

In this section we present a product formula for the solutions of
the non-auto\-no\-mous Cauchy problem \eqref{eq:perturbed}. In the
case $B(t)\equiv0$, this formula essentially goes back to Kato
\cite{Kato70}. This splitting-type formula is especially useful if
for every time $r\in\RR$ we are able to solve effectively the
\emph{autonomous} Cauchy problems
\begin{align*}
\dt u(t)&=A(r)u(t) \tag{Eq.~1}\\
\dt v(t)&=B(r)v(t) \tag{Eq.~2}
\end{align*}
with appropriate initial conditions. This is usually the case if the
operators $A(\cdot)$ and $B(\cdot)$ are partial differential
operators with time dependent coefficients or time dependent
multiplication operators. Formally, this means that we assume that
the operators $A(r)$ and $B(r)$ generate strongly continuous
operator semigroups, which we denote by using the exponential
notation as $\ee^{\cdot A(r)}$ and $\ee^{\cdot B(r)}$, respectively.
We devote this section to the simplest product formula arising from
the  \emph{sequential splitting}.

Suppose we want to determine the solution of $\mathrm{(NCP)}$ at
time $t+s>0$ and hence take the time-step $\Dt=t/n$. We start with
the known initial value $u^\mathrm{sq}(s)=x$, then solve the first
(Eq.~1) equation on the time interval $[s,s+\Dt]$ taking $r=s$. Then
we take the result $u^{(1)}_1(s+\Dt)$ as the initial value for the
second equation (Eq.~2) which we solve again on $[s,s+\Dt]$. With
this new result $u^\mathrm{sq}(s+\Dt):=u_2^{(1)}(s+\Dt)$ as initial
value for (Eq.~1) we restart the procedure and iterate it $n$ times.
Formally:
\begin{align*}
&\left\{
\begin{aligned}
\tfrac{\dd}{ \dd t}u_1^{(k)}(t)&=A(s+(k-1)\Dt)u_1^{(k)}(t), \qquad
t\in \big(s+(k-1)\Dt,s+k\Dt\big], \smallskip \\
u_1^{(k)}(s+(k-1)\Dt)&=u^\mathrm{sq}(s+(k-1)\Dt), \smallskip \\
\end{aligned} \right.\\
&\left\{ \begin{aligned} \tfrac{\dd}{\dd t}
u_2^{(k)}&=B(s+(k-1)\Dt)u_2^{(k)}(t), \qquad
t\in \big(s+(k-1)\Dt,s+k\Dt\big], \smallskip \\
u_2^{(k)}(s+(k-1)\Dt)&=u_1^{(k)}(s+k\Dt),
\end{aligned} \right.\\
&\hskip3.7em\begin{aligned}
u^\mathrm{sq}(s+k\Dt)&:=u_2^{(k)}(s+k\Dt),\end{aligned}
\end{align*}
with $k=1,2,\ldots,n$. Using that for $r\in[0,\Dt]$,
\begin{equation*}
u_1^{(k)}(s+(k-1)\Dt+r)=\ee^{r
A(s+(k-1)\Dt)}u^\mathrm{sq}(s+(k-1)\Dt),
\end{equation*}
and that
\begin{align*}
u_2^{(k)}(s+(k-1)\Dt+r)=& \ee^{rB(s+(k-1)\Dt)}u_1^{(k)}(s+k\Dt) \\
=& \ee^{rB(s+(k-1)\Dt)}\ee^{\Dt
A(s+(k-1)\Dt)}u^\mathrm{sq}(s+(k-1)\Dt),
\end{align*}
we see by a simple induction argument that the  split solution
$u^\mathrm{sq}(s+k\Dt)$, obtained by applying the sequential
splitting procedure, can be written as
\begin{equation}
u^\mathrm{sq}(s+k\Dt)=\prod_{p=0}^{k-1} \ee^{\Dt B(s+p\Dt)}\ee^{\Dt
A(s+p\Dt)}x \qquad \mbox{for } k\in\mathbb{N},\, k\Dt\leq t, \mbox{
and } x\in X. \label{spl_sq_1}
\end{equation}
In what follows, we study the convergence of this expression.
\begin{assumptions}\label{ass:1} Suppose that
\begin{enumerate}[a)]
\item the non-autonomous Cauchy problem corresponding to the operators $(A(\cdot)+ B(\cdot))$ is well-posed,
\item \bdef{(Stability)} the operators $A(r)$ and $B(r)$ are generators of $C_0$-semigroups $\ee^{\cdot A(r)}$, $\ee^{\cdot B(r)}$ of type $(M,\omega)$ ($M\geq 1$ and $\omega\in \RR$) on the Banach space $X$ and, therefore,
\begin{equation*}
(\omega, \infty) \subset \rho(A(r))\cap\rho(B(r)) \quad \mbox{for
all } r \in \RR.
\end{equation*}
Moreover, let
\begin{equation*} \label{stab}
\sup_{s\in \RR}\Bigl\| \prod_{p=n}^{1} \bigl(
\ee^{\frac{t}{n}B(s-\frac{pt}{n})}\ee^{\frac{t}{n}A(s-\frac{pt}{n})}\bigr)
\Bigr\| \leq M \ee^{\omega t}, \text{ and }
\end{equation*}
\item (\bdef{Continuity}) the maps
\begin{equation*}
t\mapsto R(\lambda,A(t))x,\qquad t\mapsto R(\lambda,B(t))x
\end{equation*}
are continuous for all $\lambda>\omega$ and $x\in X$.
\end{enumerate}
\end{assumptions}

We denote the evolution family solving \eqref{eq:perturbed} by $W$
and the corresponding evolution semigroup, generated by the closure
$\Bar\calC$ of $\calC := -\frac{\dd}{\dd s} + A(\cdot) + B(\cdot)$,
by $\calW$.

As we shall see in a moment, Assumption \ref{ass:1} yields that the
multiplication operators $A(\cdot)$, $B(\cdot)$ with appropriate
domain generate strongly continuous multiplication semigroups on
$\C_0(\RR;X)$ (for more on this matter we refer to Engel and Nagel
\cite[Sec.~III.4.13]{Engel-Nagel} and Graser \cite{Graser}).

\begin{theorem}\label{th:firstproduct}
Under Assumption \ref{ass:1} one has the convergence
\begin{equation}
\label{eq:ev_product} W(t,s)x= \lim_{n\to\infty} \prod_{p=0}^{n-1}
\bigl(\ee^{\frac{t-s}{n}B(s+
\frac{p(t-s)}{n})}\ee^{\frac{t-s}{n}A(s+\frac{p(t-s)}{n})} \bigr)x
\end{equation}
for all $x\in X$, locally uniformly in $s,t$ with $s\leq t$.
\end{theorem}

\begin{proof}
The main idea of the proof is analogous to the one in Nickel
\cite[Proposition 3.2]{nickel00}. Consider the  semigroups
$\ee^{\cdot A(r)}$ and $\ee^{\cdot B(r)}$ for given $r\in \RR$. By
the uniform growth assumption in \ref{ass:1}.b) on the semigroups,
for fixed $t\geq 0$ the function $r\mapsto \ee^{tA(r)}f(r)$ vanishes
at infinity whenever $f$ has this property. We also have that  the
function $r\mapsto \ee^{tA(r)}$ is strongly continuous. Indeed,  by
the Trotter-Kato Theorem (see Engel and Nagel
\cite[Thm.~III.4.8]{Engel-Nagel}) we even obtain that
$\RR_+\times\RR \ni (t,r)\mapsto \ee^{tA(r)}$ is strongly
continuous. All these reasonings are, of course, true if $A(r)$ is
replaced by $B(r)$. Let now $f\in\BUC(\RR;X)$. Then $r\mapsto
\ee^{tA(r)}f(r)$ is continuous, too. We have therefore shown that the
multiplication semigroups $\ee^{tA(\cdot)}$ and $\ee^{tB(\cdot)}$,
generated by
 the multiplication operators $A(\cdot)$ and $B(\cdot)$, both act on the space
 $\calX= \C_0(\RR;X)$, see also Graser \cite{Graser}.
It can be seen by induction that
\begin{equation*}
\bigl(\mathcal R \bigl(\tfrac{t}{n}\bigr) \ee^{\frac{t}{n}B(\cdot)}
\ee^{\frac{t}{n}A(\cdot)} \bigr)^n f(\cdot) = \prod_{p=n}^{1}
\bigl(\ee^{\frac{t}{n}B(\cdot - \frac{pt}{n})}
\ee^{\frac{t}{n}A(\cdot - \frac{pt}{n})}\bigr)\mathcal R
(t)f(\cdot).
\end{equation*}

The stability assumption \ref{ass:1}.b) immediately implies the
stability for the finite difference method
$F(h):=\calR(h)\ee^{hB(\cdot)}\ee^{hA(\cdot)} $. Consistency is
standard to check:  take $f\in \calX_1\cap \dom(A(\cdot))\cap
\dom(B(\cdot))$. Then we can write 
\begin{align*}
\lim_{h\downarrow 0}\frac{F(h)f-f}{h}&=\lim_{h\downarrow 0}\Bigl[\calR(h)\ee^{hB(\cdot)}\frac{ \ee^{hA(\cdot)}f-f}{h}+\calR(h)\frac{\ee^{hB(\cdot)}f-f}{h}+\frac{\calR(h)f-f}{h}\Bigr]\\
&=A(\cdot)f+B(\cdot)f-f'.
\end{align*}

By our well-posedness assumptions, the closure of the operator
$\mathcal{C} = -\frac{\dd}{\dd s} + B(\cdot) + A(\cdot)$ generates a
strongly continuous semigroup on $\calX$, hence the set
$(\lambda-\calC)\dom(\calC)$ is dense in $\calX$. By the stability
assumption we can apply Chernoff's Theorem \ref{thm:chernoff} with
the three operators $-\frac{\dd}{\dd s}$, $A(\cdot)$, $B(\cdot)$,
and obtain that the evolution semigroup generated by $\overline
\calC$ is given by
\begin{equation*}
\calW(t)f = \lim_{n\to\infty} \prod_{p=n}^{1}
\bigl(\ee^{\frac{t}{n}B(\cdot - \frac{pt}{n})}
\ee^{\frac{t}{n}A(\cdot - \frac{pt}{n})}\bigr)f(\cdot-t).
\end{equation*}
The above limit is to be understood in the topology of $\calX$, that
is, in the uniform topology. By using this, and by applying the
formula \eqref{eq:recovery} from the previous section, we can
recover the evolution family from the evolution semigroup and arrive
at the formula
\begin{equation*}
W(t,s)x= \lim_{n\to\infty} \prod_{p=n}^{1}
\bigl(\ee^{\frac{t-s}{n}B(t - \frac{p(t-s)}{n})}
\ee^{\frac{t-s}{n}A(t - \frac{p(t-s)}{n})}\bigr)x,
\end{equation*}
from which the assertion follows.
\end{proof}

\begin{remark}
In the proof of Theorem \ref{th:firstproduct} we have used that the
semigroups $\ee^{\cdot A(r)}$ and $\ee^{\cdot B(r)}$ map
$\C_0(\RR;X)$ into itself. If $\ee^{\cdot A(r)}$ and $\ee^{\cdot
B(r)}$ are uniformly strongly continuous in $r\in \RR$, then one
could also work  on the space $\calX=\BUC(\RR;X)$.
\end{remark}

\begin{remark}
The stability condition b) is automatically satisfied, if $A(t)$ and
$B(t)$ are generators of quasi-contractive semigroups with uniform
exponential bound $\omega$ for all $t$.
\end{remark}

\begin{remark}
In Vuillermot et al.~\cite{Vuillermot_etal,VWZ2009}, the authors
prove the representation formula \eqref{eq:ev_product} where $A(t)$
and $B(t)$ are generators of contraction semigroups, the family
$A(\cdot)$ satisfies a version of the so-called parabolic condition
and the family $B(\cdot)$ is a small perturbation. Theorem
\ref{th:firstproduct} can be seen as a generalization of this result
and can be applied not only in a larger class of parabolic problems
but also in the hyperbolic case. In \cite{Vuillermot2010}
Vuillermot proves a Chernoff-type approximation theorem for
time-dependent operator families. Under appropriate consistency and stability
assumptions it is possible to derive formula \eqref{eq:ev_product}
from this result (as done in \cite{Vuillermot2010}) instead of
proving it by the application of the classical Chernoff's Theorem to
evolution semigroups. It is however amongst our aims to emphasize
that semigroup techniques may be used to prove approximation
results also for non-autonomous problems.
\end{remark}

\begin{remark}
In case $B(t)\equiv0$, we recover the well-known representation
formula
\begin{equation*}
U(t,s)x = \lim_{n\to\infty} \prod_{p=0}^{n-1} \ee^{\frac{t-s}{n}A(s+
\frac{p(t-s)}{n})} x,
\end{equation*}
see Nickel \cite[Proposition 3.2]{nickel00} and Schnaubelt
\cite[Theorem 2.1]{Schnaubelt99}. Again, the stability condition
reduces essentially to the classical stability condition of Kato
\cite{Kato70}.
\end{remark}

\begin{remark}
It is straightforward to check that if one of the equations is
autonomous, e.g., $A(t)\equiv A$, then we arrive at the same product
formula but we can split the original operator $\calC$ into two (and
not three) operators, namely into $-\frac{\dd}{\dd s} + A$ and
$B(\cdot)$.
\end{remark}


\section{Operator splitting}\label{sec:5}
In this section we assume that we can solve the
\emph{non-autonomous} equations
\begin{align}
\tag{Eq.~A}\label{Eq1} \dt u(t)&=A(t)u(t), \\
\tag{Eq.~B}\label{Eq2} \dt v(t)&=B(t)v(t)
\end{align}
and want to construct the solution of \eqref{eq:perturbed} applying an operator splitting procedure. For the sake of simplicity we
only present the case of sequential splitting: We start with the
initial value $u^\mathrm{sq}(s)=x$, then solve the first equation on
the time interval $[s,s+\Dt]$. Then we take this $u^{(1)}_1(s+\Dt)$
as the initial value for the second equation which we solve on
$[s,s+\Dt]$. With this result
$u^\mathrm{sq}(s+\Dt):=u_2^{(1)}(s+\Dt)$ as initial value for
\eqref{Eq1} we restart the procedure and iterate it $n$ times.
Formally:
\begin{align*}
&\left\{
\begin{aligned}
\tfrac{\dd}{\dd t}u_1^{(k)}(t)&=A(t)u_1^{(k)}(t), \qquad
t\in \big(s+(k-1)\Dt,s+k\Dt\big], \smallskip \\
u_1^{(k)}(s+(k-1)\Dt)&=u^\mathrm{sq}(s+(k-1)\Dt), \smallskip \\
\end{aligned} \right.\\
&\left\{ \begin{aligned} \tfrac{\dd}{\dd
t}u_2^{(k)}(t)&=B(t)u_2^{(k)}(t), \qquad
t\in \big(s+(k-1)\Dt,s+k\Dt\big], \smallskip \\
u_2^{(k)}(s+(k-1)\Dt)&=u_1^{(k)}(s+k\Dt),
\end{aligned} \right.\\
&\hskip3.7em\begin{aligned}
u^\mathrm{sq}(s+k\Dt)&:=u_2^{(k)}(s+k\Dt),\end{aligned}
\end{align*}
for $k=1,2,\ldots,n$. If $U$ and $V$ denote the evolution families
solving the above equations (Eq.~A)-(Eq.~B), then we have
 $$u_1^{(k)}(r)=U(r,s+(k-1)\Dt)u^\mathrm{sq}(s+(k-1)\Dt),$$ and
\begin{align*}
u_2^{(k)}(r)&=V(r,s+(k-1)\Dt)u_1^{(k)}(s+k\Dt)\\
&=V(r,s+(k-1)\Dt)U(s+k\Dt,s+(k-1)\Dt)u^\mathrm{sq}(s+(k-1)\Dt).
\end{align*}
By this the  splitting solution $u^\mathrm{sq}$ can be written as
\begin{equation*}
u^\mathrm{sq}(s+k\Dt)=\prod_{p=0}^{k-1}\Bigl(V(s+(p+1)\Dt,s+p\Dt)U(s+(p+1)\Dt,s+p\Dt)\Bigr)x.
\end{equation*}
In the following we analyze the convergence of this procedure.

\begin{assumptions}\label{ass:2}
Suppose that
\begin{enumerate}[a)]
\item the non-autonomous Cauchy problems corresponding to the operators $A(\cdot)+ B(\cdot)$, $A(\cdot)$, and $B(\cdot)$ are well-posed, and that
\item \bdef{(Stability)} there exist $M\geq 1$ and $\omega\in\RR$ such that
\begin{equation*}
\sup_{s\in\RR}\Bigl\| \prod_{p=n-1}^{0}V\bigl(s-\tfrac{pt}{n},s -
\tfrac{(p+1)t}{n}\bigr) U\bigl(s - \tfrac{pt}{n},s -
\tfrac{(p+1)t}{n}\bigr)\Bigr\| \leq M\ee^{\omega t}.
\end{equation*}
\end{enumerate}
\end{assumptions}

Here, again, the evolution family solving the Cauchy problem corresponding
to $A(\cdot)$ and $B(\cdot)$, will be denoted by $U$ and $V$,
respectively. Further, we denote the evolution family solving
\eqref{eq:perturbed} by $W$ and the corresponding evolution
semigroup, generated by the closure of $\mathcal{C} =
-\frac{\dd}{\dd s} + A(\cdot) + B(\cdot)$, by $\calW$.

\begin{theorem}\label{thm:evolfamkonv}
Under Assumptions \ref{ass:2} one has the convergence
\begin{equation*}
W(t,s)x\hskip-0.2em  = \hskip-0.3em \lim_{n\to\infty}\hskip-0.1em
\prod_{p=0}^{n-1}\hskip-0.3em  V  (s + \tfrac{(p+1)(t-s)}{n},s +
\tfrac{p(t-s)}{n}) \hskip-0.1em U ( s+ \tfrac{(p+1)(t-s)}{n}, s +
\tfrac{p(t-s)}{n})x
\end{equation*}
for all $x\in X$.
\end{theorem}

\begin{proof}
In the space $\mathcal{X}$, we define
\begin{align*}
\mathcal F(t) &:= V(\cdot, \cdot - t)\\
\mbox{and}\quad\quad\quad\mathcal G(t) &:= U(\cdot, \cdot -
t)\mathcal R(t).
\end{align*}
Inductively, one can see that
\begin{align*}
\bigl(\mathcal F(\tfrac{t}{n})\mathcal G(\tfrac{t}{n})\bigr)^{n} f &= \bigl(V(\cdot, \cdot-\tfrac{t}{n}) U(\cdot, \cdot-\tfrac{t}{n})\mathcal R(\tfrac tn)\bigr)^n f \\
&= \prod_{p=n-1}^{0}V\bigl(\cdot-\tfrac{pt}{n},\cdot -
\tfrac{(p+1)t}{n}\bigr) U\bigl(\cdot - \tfrac{pt}{n},\cdot -
\tfrac{(p+1)t}{n}\bigr)f(\cdot-t).
\end{align*}

By our assumptions, the closure $\overline\calC$ of the operator
$\mathcal{C} = -\frac{\dd}{\dd s} + A(\cdot) + B(\cdot)$ generates a
strongly continuous semigroup on $\calX$, and hence the set
$(\lambda-\calC)\dom(\calC)$ is dense. Straightforward calculation
analogous to the one in the proof of Theorem \ref{th:firstproduct}
yields that $\bigl(\mathcal F(\cdot) \mathcal
G(\cdot)\bigr)'(0)f=\calC f$ for $f\in \dom(\calC)$. Hence, by the
stability assumption, we can apply Chernoff's Theorem to this
function and obtain that the evolution semigroup generated by
$\overline \calC$ is given by
\begin{align*}
\calW(t)f = \lim_{n\to\infty}
\prod_{p=n-1}^{0}V\bigl(\cdot-\tfrac{pt}{n},\cdot -
\tfrac{(p+1)t}{n}\bigr) U\bigl(\cdot - \tfrac{pt}{n},\cdot -
\tfrac{(p+1)t}{n}\bigr)f(\cdot-t).
\end{align*}
From this, by picking some $f\in \calX$ with $f(s)=x$, we obtain for
the evolution family
\begin{align*}
&W(t,s)x =\\
&=\lim_{n\to\infty} \prod_{p=n-1}^{0}V\bigl(t-\tfrac{p(t-s)}{n},t - \tfrac{(p+1)(t-s)}{n}\bigr) U\bigl(t - \tfrac{p(t-s)}{n},t - \tfrac{(p+1)(t-s)}{n}\bigr)x \\
&= \lim_{n\to\infty} \prod_{p=0}^{n-1} V \bigl (s +
\tfrac{(p+1)(t-s)}{n},s + \tfrac{p(t-s)}{n}\bigr) U \bigl( s+
\tfrac{(p+1)(t-s)}{n}, s + \tfrac{p(t-s)}{n}\bigr)x,
\end{align*}
which was to be proved.
\end{proof}

\begin{remark}
Note that the stability condition is trivially satisfied if the
evolution families $U$ and $V$ are quasi-contractive, i.e., if
$M\leq 1$ can be taken in Definition \ref{evf} (iii). In general, as
usual with stability assumptions, it is rather hard to verify.
\end{remark}

Using similar arguments but a different decomposition, we arrive at
a different splitting formula using evolution families corresponding
to different (time-rescaled) evolution equations.

\begin{proposition}
Suppose that the operator families $A(\cdot/2)$, $B(\cdot/2)$ and
$A(\cdot)+B(\cdot)$ generate the evolution families $\widetilde U$,
$\widetilde V$ and $W$, respectively. Assume furthermore that there
is $M\geq 1$ and $\omega\in\RR$ such that
\begin{equation*}
\sup_{s\in\RR}\Bigl\| \prod_{p=n-1}^{0}\widetilde
V\bigl(2s-\tfrac{2pt}{n},2s - \tfrac{(2p+1)t}{n}\bigr) \widetilde
U\bigl(2s - \tfrac{(2p+1)t}{n},2s - \tfrac{(2p+2)t}{n}\bigr)\Bigr\|
\leq M \ee^{\omega t}.
\end{equation*}
Then we have
\begin{align*}
&W(t,s)x \\&= \lim_{n\to\infty} \prod_{p=0}^{n-1} \widetilde V \bigl
(2s + \tfrac{2(p+1)(t-s)}{n},2s + \tfrac{(2p+1)(t-s)}{n}\bigr)
\widetilde U \bigl( 2s+ \tfrac{(2p+1)(t-s)}{n}, 2s +
\tfrac{2p(t-s)}{n}\bigr)x.
\end{align*}
\end{proposition}

\begin{proof}
In the space $\mathcal{X}$, we write formally
\begin{equation*}
-\frac{\dd}{\dd s} + A(\cdot) + B(\cdot) = \bigl(-\frac{\dd}{2\dd s}
+ A(\cdot)\bigr) + \bigl(-\frac{\dd}{2\dd s} + B(\cdot)\bigr) =
\calA_1+ \calB_1.
\end{equation*}
Since the division by 2 in the formula means a rescaling of the
corresponding evolution semigroups $\calS$ and $\calT$, we obtain
the representation formulas
\begin{align*}
\calS(t) &= \widetilde V(2\cdot, 2\cdot - t)\calR(t/2)\\
\calT(t) &= \widetilde U(2\cdot, 2\cdot - t)\calR(t/2).
\end{align*}
By induction  one can see that
\begin{align*}
&\bigl(\calS(\tfrac{t}{n})\calT(\tfrac{t}{n})\bigr)^n f = \bigl(\widetilde V(2\cdot, 2\cdot-\tfrac{t}{n})\mathcal R(t/2n) \widetilde U(2\cdot, 2\cdot-\tfrac{t}{n})\mathcal R(t/2n)\bigr)^n f \\
&\quad= \prod_{p=n-1}^{0}\widetilde
V\bigl(2\cdot-\tfrac{2pt}{n},2\cdot - \tfrac{(2p+1)t}{n}\bigr)
\widetilde U\bigl(2\cdot - \tfrac{(2p+1)t}{n},2\cdot -
\tfrac{(2p+2)t}{n}\bigr)f(\cdot-t).
\end{align*}
Again, the closure $\overline\calC$ of the operator $\mathcal{C} =
-\frac{\dd}{\dd s} + A(\cdot) + B(\cdot)$ generates a strongly
continuous semigroup on $\calX$, hence $(\lambda-\calC)\dom(\calC)$
is dense. By this and by the stability assumption  Chernoff's
Theorem is applicable. We obtain that the evolution semigroup
generated by $\overline \calC$ is given by
\begin{equation*}
\calW(t)f = \lim_{n\to\infty} \prod_{p=n-1}^{0}\widetilde
V\bigl(2\cdot-\tfrac{2pt}{n},2\cdot - \tfrac{(2p+1)t}{n}\bigr)
\widetilde U\bigl(2\cdot - \tfrac{(2p+1)t}{n},2\cdot -
\tfrac{(2p+2)t}{n}\bigr)f(\cdot-t).
\end{equation*}
By passing to the evolution family we get the assertion:
\begin{align*}
&W(t,s)x \\
&= \lim_{n\to\infty} \prod_{p=n-1}^{0}\hskip-0.5em\widetilde V\bigl(2t-\tfrac{2p(t-s)}{n},2t - \tfrac{(2p+1)(t-s)}{n}\bigr) \widetilde U\bigl(2t - \tfrac{(2p+1)(t-s)}{n},2t - \tfrac{(2p+2)(t-s)}{n}\bigr)x \\
&= \lim_{n\to\infty} \prod_{p=0}^{n-1} \widetilde V \bigl (2s +
\tfrac{2(p+1)(t-s)}{n},2s + \tfrac{(2p+1)(t-s)}{n}\bigr) \widetilde
U \bigl( 2s+ \tfrac{(2p+1)(t-s)}{n}, 2s + \tfrac{2p(t-s)}{n}\bigr)x.
\end{align*}
\end{proof}

\begin{remark}
Note that, in contrast to the autonomous case, there is no general
connection between the evolution families $U$ and $\widetilde U$,
see Nickel \cite{nickel97}.
\end{remark}


\section{Generalizations and Remarks}\label{sec:6}

\subsection*{Higher order splitting methods}
We now show how the previous results generalize to higher order
splitting methods. The results are, using the stage set up
previously, direct applications of the corresponding autonomous
results applied to the evolution semigroups. We restrict ourselves
to the Strang and symmetrically weighted splitting, but other
splitting methods can be handled analogously. In any case only the
stability condition has to be adapted. This, however, is always
satisfied (and typically verifiable) if the operators involved are
contractions.

\begin{theorem}
Suppose that Assumptions \ref{ass:1} a) and c) are satisfied, and
that the stability condition holds in the following form:
\begin{enumerate}[b')]
\item \rule{0pt}{0pt}\hfill$\displaystyle\sup_{s\in\RR}\Bigl\|\prod_{p=n-1}^{0} \ee^{\frac{t}{2n}A(s-\frac{pt}{n})}\ee^{\frac{t}{n}B(s-\frac{pt}{n})} \ee^{\frac{t}{2n}A(s-\frac{pt}{n})}\Bigr\|\leq M\ee^{\omega t}$\hfill\rule{0pt}{0pt}
\end{enumerate}
in the case of the Strang splitting, or:
\begin{enumerate}[b'')]
\item \rule{0pt}{0pt}\hfill$\displaystyle\sup_{s\in \RR}\frac{1}{2^n}\Bigl\|\prod_{p=n-1}^{0} \Bigl( \ee^{\frac{t}{n}A(s-\frac{pt}{n})} \ee^{\frac{t}{n}B(s-\frac{pt}{n})} + \ee^{\frac{t}{n}B(s-\frac{pt}{n})}\ee^{\frac{t}{n}A(s-\frac{pt}{n})} \Bigr)\Bigr\|\leq M\ee^{\omega t}$\hfill\rule{0pt}{0pt}
\end{enumerate}
in the case of the symmetrically weighted splitting. Then we have
\begin{equation*}
W(t,s)x= \lim_{n\to\infty} \prod_{p=0}^{n-1}
\ee^{\frac{t-s}{2n}A(s+\frac{p(t-s)}{n})}\ee^{\frac{t-s}{n}B(s+\frac{p(t-s)}{n})}
\ee^{\frac{t-s}{2n}A(s+\frac{p(t-s)}{n})}x
\end{equation*}
for all $x\in X$ in case of the Strang splitting; and we have
\begin{align*}
&W(t,s)x\\
&= \lim_{n\to\infty} \frac{1}{2^n}\prod_{p=0}^{n-1} \Bigl(
\ee^{\frac{t-s}{n}A(s+\frac{p(t-s)}{n})}
\ee^{\frac{t-s}{n}B(s+\frac{p(t-s)}{n})} +
\ee^{\frac{t-s}{n}B(s+\frac{p(t-s)}{n})}\ee^{\frac{t-s}{n}A(s+\frac{p(t-s)}{n})}
\Bigr)x
\end{align*}
for all $x\in X$ in case of the symmetrically weighted splitting.
\end{theorem}
\begin{proof}
The statements follow immediately by the same reasonings as in the
proof of Theorem \ref{th:firstproduct}, but now considering the
expressions
\begin{equation*}
\Bigl(\mathcal R \bigl(\tfrac{t}{n}\bigr)\ee^{\frac{t}{2n}A(\cdot)}
\ee^{\frac{t}{n}B(\cdot)} \ee^{\frac{t}{2n}A(\cdot)} \Bigr)^n
\end{equation*}
for the Strang-splitting, and
\begin{equation*}
\frac{1}{2^n}\Bigl(\mathcal R
\bigl(\tfrac{t}{n}\bigr)\bigl(\ee^{\frac{t}{n}A(\cdot)}\ee^{\frac{t}{n}B(\cdot)}+\ee^{\frac{t}{n}B(\cdot)}
\ee^{\frac{t}{n}A(\cdot)} \bigr)\Bigr)^n,
\end{equation*}
for the weighted splitting, respectively.
\end{proof}

\begin{remark}
It can be shown by exactly the same arguments as in Csom\'{o}s and Nickel \cite[Lemma 2.3]{Csomos-Nickel} that the stability condition $(b')$ is equivalent to the stability condition in Assumption \ref{ass:1} b) for the sequential splitting.
\end{remark}

\vskip 5pt

\subsection*{Spatial approximations}
~~Continuing earlier investigations started in B\'{a}tkai, Csom\'{o}s and
Nickel \cite{Batkai-Csomos-Nickel}, we show that operator splitting
combined with spatial approximations is also convergent. We only
concentrate on the formula \eqref{eq:ev_product} for the sequential
splitting. Other methods can be considered analogously.

\begin{assumptions}\label{ass:proj_int}
Let $X_m$, $m\in\mathbb{N}$ be Banach spaces and take operators
\begin{equation*}
P_m:\ X\rightarrow X_m \qquad \mbox{and} \qquad J_m:\ X_m\rightarrow
X \nonumber
\end{equation*}
fulfilling the following properties:
\renewcommand{\labelenumi}{(\roman{enumi})}
\begin{enumerate}
\item $P_mJ_m=\Id_m$ for all $m\in\mathbb{N}$, where $\Id_m$ is the identity operator in $X_m$,
\item $\lim\limits_{m\rightarrow\infty}J_mP_mx=x$ for all $x\in X$,
\item $\|J_m\|\le K$ and $\|P_m\|\le K$ for all $m\in\mathbb{N}$ and a suitable absolute constant $K\geq 1$.
\end{enumerate}
\label{def:proj}
\end{assumptions}

The operators $P_m$ together with the spaces $X_m$ usually refer to
a kind of spatial discretization method (triangulation, Galerkin
approximation, Fourier coefficients, etc.), the spaces $X_m$ are in
most applications finite dimensional spaces, and the operators $J_m$
refer to the interpolation method describing how we associate
specific elements of the function space to the elements of the
approximating spaces (linear/polynomial/spline interpolation, etc.).

\begin{assumptions}\label{ass:stabcons}\hskip-0.1em
For each $m\in \NN$ and  $r\in \RR$ let the operators
$A_m(r)$ and $B_m(r)$ be generators of
strongly continuous semigroups $\ee^{t A_m(r)}$ and $\ee^{t
B_m(r)}$, respectively. Assume furthermore that
\begin{enumerate}[a)]
\item (\bdef{Stability}) there exist constants $M\ge 1$ and $\omega\in\mathbb{R}$ such that
\begin{equation*}
 \|\ee^{hA(r)}\|, \, \|\ee^{hA_m(r)}\|,\, \|\ee^{hB(r)}\|,\,\|\ee^{hB_m(r)}\| \le M\ee^{\omega h}, \text{ for all } h>0 \text{ and } r\in\RR, \text{ that }
\end{equation*}
\begin{equation*} \label{stab_approx}
\sup_{s\in\RR}\Bigl\| \prod_{p=n-1}^{0}
\bigl(\ee^{\frac{t}{n}B_m(s-\frac{pt}{n})}
\ee^{\frac{t}{n}A_m(s-\frac{pt}{n})}\bigr) \Bigr\| \leq M
\ee^{\omega t}, \text{ and that }
\end{equation*}
\item (\bdef{Consistency}) the identities
$\lim\limits_{m\to\infty}J_mA_m(\cdot)P_m f=A(\cdot)f$ for all $f\in \dom(A(\cdot))$, and
 $\lim\limits_{m\to\infty}J_mB_m(\cdot)P_m f=B(\cdot)f$ for all $f\in \dom(B(\cdot))$ hold.
\end{enumerate}
\end{assumptions}

As in B\'{a}tkai, Csom\'{o}s and Nickel \cite{Batkai-Csomos-Nickel}, stability
and consistency implies convergence.

\begin{theorem}\label{th:product_approx}
Suppose that Assumption \ref{ass:stabcons} is satisfied. Then one
has the convergence
\begin{equation*}
W(t,s)x= \lim_{m\to\infty}\lim_{n\to\infty}J_m \prod_{p=0}^{n-1}
\bigl(\ee^{\frac{t-s}{n}B_m(s+\frac{p(t-s)}{n})}
\ee^{\frac{t-s}{n}A_m(s+\frac{s+p(t-s)}{n})}\bigr) P_mx
\end{equation*}
for all $x\in X$.
\end{theorem}

\begin{proof}
We will apply B\'{a}tkai, Csom\'{o}s and Nickel \cite[Theorem
3.6]{Batkai-Csomos-Nickel}, the modified Chernoff's Theorem  directly.
To this end, define the spaces
\begin{equation*}
\calX_m=\C_0(\RR;X_m), \quad \calX:=\C_0(\RR;X)
\end{equation*}
and the projection operators
\begin{equation*}
\calP_m=\Id\otimes P_m: \calX\to \calX_m, \quad
(\calP_mf)(t):=P_mf(t),
\end{equation*}
and interpolation operators
\begin{equation*}
\calJ_m=\Id\otimes J_m: \calX_m\to \calX, \quad
(\calJ_mf_m)(t):=J_mf_m(t).
\end{equation*}
We have to check that these operators satisfy the conditions in
Assumption \ref{ass:proj_int}. Conditions (i) and (iii) are
immediate from the definitions. The $(\calJ_m\calP_m f)(s)\to f(s)$
is true pointwise. We have to show that the convergence holds in fact uniformly
in $s\in \RR$. Take $\varepsilon >0$. Let $f\in \calX$ and
$[a,b]\subset \RR$ such that $\|f(s)\| \leq
\tfrac{\varepsilon}{2K^2}$ for all $s\in \RR\setminus [a,b]$. Then
\begin{equation*}
\|J_mP_mf(s)-f(s)\|\leq \varepsilon
\end{equation*}
for $s\in \RR\setminus [a,b]$.
Since $f$ is uniformly continuous, there is $\delta>0$ such that for
all $s,t\in [a,b]$, $|s-t|<\delta$, we have $\|f(s)-f(t)\|\leq
\tfrac{\varepsilon}{K^2+2}$. Take a partition $a=s_0<s_1<\ldots
<s_n=b$ such that $|s_{i+1}-s_i|<\delta$. Then by definition, there
is $M>0$ such that for all $m\geq M$
\begin{equation*}
\|J_mP_mf(s_i)-f(s_i)\|\leq \tfrac{\varepsilon}{K^2+2}.
\end{equation*}
Since for $s\in [a,b]$ there is $j$ such that $s\in
[s_{j},s_{j+1}]$, we get for $m\geq M$,
\begin{align*}
&\|J_mP_mf(s)-f(s)\| \\
&\quad\quad\leq \|J_mP_m(f(s)-f(s_j))\| + \|J_mP_m f(s_j)-f(s_j)\| +
\|f(s_j)-f(s)\| \leq \varepsilon.
\end{align*}
Hence, $\|\calJ_m\calP_mf-f\|_\infty\leq \varepsilon$ holds for all
$m\geq M$.

\medskip
The validity of Assumption \ref{ass:stabcons}  implies that the
corresponding multiplication semigroups satisfy the necessary
stability and consistency conditions.
\end{proof}


\subsection*{Positivity preservation}
~~As it was pointed out by W. Arendt (Ulm), the product and splitting formulas can be used to show positivity properties of evolution families. On the terminology and properties of positive operator semigroups see Arendt et al. \cite{Arendt_etal} or Engel and Nagel \cite[Section VI.1]{Engel-Nagel}.

\begin{theorem} Assume that $X$ is a Banach lattice.
\begin{enumerate}
\item Let the conditions of  Assumptions \ref{ass:1} are satisfied and that all the operators $A(r)$ and $B(r)$ generate positive semigroups. Then the evolution family $W$ given by \eqref{eq:ev_product} in Theorem \ref{th:firstproduct} is positive.
\item Let the conditions of Assumptions \ref{ass:2} are satisfied and that all the evolution families $U$ and $V$  are positive. Then the evolution family $W$ given by Theorem \ref{thm:evolfamkonv} is positive.
\end{enumerate}
\end{theorem}

The proof is an immediate consequence of the fact that the corresponding multiplication, shift, and evolution semigroups are positive. It would be an important and interesting question whether similar results hold for shape preserving semigroups in the sense of Kov\'{a}cs \cite[Definition 20]{Kovacs}.


\section{A non-autonomous parabolic equation}\label{sec:7}

In order to demonstrate the range of our results, we will consider
    an important and much studied parabolic equation
\begin{equation}\label{eq:schroedinger}
\partial_t u(x,t) = \Delta u(x,t) + V(x,t)u(x,t)
\end{equation}
in $\RR^d$ with appropriate initial conditions, where $V$ is a
smooth and bounded function. Rewritten abstractly this takes the
form
\begin{equation}\label{eq:abstsch}
\tfrac {\dd}{\dd t}u(t)=\Delta u(t)+V(t)u(t)
\end{equation}
with  $u:\RR_+\to \Ell^2(\RR^d)=:X$ a vector valued function. Hence
a straightforward choice for the splitting for the evolution
semigroups is
\begin{equation*}
\mathcal{A}:=-\tfrac{\dd}{\dd s}+\Delta,\quad \mathcal{B}:=\mbox{the
pointwise multiplication by $V(t)$}.
\end{equation*}
These operators (with appropriate domain) generate the following
semigroups on the Banach space $\calX:=\BUC(\RR;\Ell^2(\RR^d))$
\begin{equation*}
[\mathcal{T}(t)f](s):=\ee^{t\Delta} f(s-t)\quad\mbox{and}\quad
[\mathcal{S}(t)f](s):=\ee^{tV(s)}f(s).
\end{equation*}
We shall assume that $V\in \BUC(\RR;\Ell^\infty(\RR^d))$, so
$\mathcal{B}$ is bounded. The domain of the generator of
$\mathcal{S}$ can be given explicitly, see Nagel, Nickel and Romanelli
\cite[Prop.~4.3]{NaNiRo}):
\begin{equation*}
\dom(\mathcal{A})=\bigl\{f\in \BUC(\RR;X)\cap
\BUC^1(\RR;X_{-1}):-f'+\Delta_{-1}f\in \BUC(\RR;X)\bigr\},
\end{equation*}
here $\Delta_{-1}$ with domain $\Ell^2(\RR^d)$ is the generator of
the extrapolated semigroup, see Engel and Nagel \cite[Section
II.5.a]{Engel-Nagel} for the corresponding definitions.

As a corollary of Theorem \ref{th:firstproduct} we obtain the
convergence of the sequential (and also the Strang) splitting
procedures.
\begin{proposition}\label{prop:strangconv}
Suppose that the potential $V\in \BUC(\RR;\Ell^\infty(\RR^d))$. Let
$\calW$ denote the semigroup generated by $\calA+\calB$ on
$\BUC(\RR;\Ell^{2}(\RR^d))$. For every function $f\in
\BUC(\RR;\Ell^2(\RR^d))$ we have the product formula
\begin{equation*}
\lim_{n\to \infty}
\bigl(\mathcal{S}(\tfrac{t}{n})\mathcal{T}(\tfrac{t}{n})\bigr)^n
f=\mathcal{W}(t)f,
\end{equation*}
where the convergence is uniform on compact time-intervals. Let
$(W(t,s))_{t\geq s}$ denote the evolution system solving
\eqref{eq:abstsch} on $\Ell^2(\RR^d)$. Then for every $u_0\in
\Ell^{2}(\RR^d)$ we have
\begin{equation*}
\lim_{n\to \infty}\Bigl\|W(t,s)u_0-
\prod_{p=0}^{n-1}
\ee^{\frac{t-s}{n}V(s+\frac{pt}{n})}\ee^{\frac{t-s}{n}\Delta
}u_0\Bigr\|=0,
\end{equation*}
locally uniformly for $s\leq t$.
\end{proposition}
\begin{proof}
For the first assertion we only have to verify the stability
Assumption \ref{ass:1} b), and then the assertion follows directly
from Chernoff's Theorem \ref{thm:chernoff}. Stability follows,
because the semigroup $\ee^{t\Delta}$ is contractive and $V(s)$ is uniformly
bounded. The second assertion is a direct consequence of Theorem
\ref{th:firstproduct}.
\end{proof}
Next we study convergence rates for the sequential splitting
procedure applied to the above equation \eqref{eq:schroedinger}. To
this end we apply Theorem \ref{Thm:Jahnke-Lubich} to the
corresponding evolution semigroups.

\begin{theorem}\label{thm:SchStrangsgrp}
Suppose that $V\in \BUC(\RR;W^{2,\infty}(\RR^d))\cap
\BUC^1(\RR;\Ell^{\infty}(\RR^d))$. If $f\in
\BUC^1(\RR;\Sob^2(\RR^d))$, we obtain
\begin{equation*}
\bigl\|\bigl(\mathcal{S}(\tfrac{t}{n})\mathcal{T}(\tfrac{t}{n})\bigr)^n-\mathcal{W}(t)f\bigr\|\leq
\frac{Ct^2}{n}\|f\|_{\BUC^1(\RR;\Sob^2(\RR^d))}.
\end{equation*}
\end{theorem}

Before we prove the theorem, let us first reformulate this product formula
for the solutions of the non-autonomous problem.
\begin{corollary}\label{thm:Strangevo}
Consider the non-autonomous parabolic equation
\begin{equation*}
\begin{cases}
\partial_t u(x,t) = \Delta u(x,t) + V(x,t)u(x,t),&t\geq s,\:x\in \RR^d,\\
u(x,s)=u_0(x),&x\in \RR^d.
\end{cases}
\end{equation*}%
Suppose that $V\in \BUC(\RR;W^{2,\infty}(\RR^d))\cap
\BUC^1(\RR;\Ell^{\infty}(\RR^d))$. If $u_0\in \Sob^2(\RR^d)$ then
for the evolution family $(W(t,s))_{t\geq s}$ solving the above
problem we have
\begin{equation*}
\Bigl\|W(t,s)u_0 - \prod_{p=0}^{n-1}
\ee^{\frac{t-s}{n}V(s+\frac{pt}{n})}\ee^{\frac{t-s}{n}\Delta }
u_0\Bigr\|\leq \frac{C(t-s)^2}{n}\|u_0\|_{H^2}.
\end{equation*}
\end{corollary}
\begin{proof}
The assertion follows from Theorem \ref{thm:SchStrangsgrp}, from the
calculations in the proof of Theorem \ref{prop:strangconv} and from
the fact that the constant function $f(s):=u_0\in \Sob^2(\RR^d)$ is
in the domain of $\calA$.
\end{proof}

\medskip In order to prove Theorem \ref{thm:SchStrangsgrp} we have to verify the commutator condition in Theorem \ref{Thm:Jahnke-Lubich} for the generators of the evolution semigroups. To  do this, we need the following abstract identification of the domains of fractional powers of evolution semigroup generators.

\medskip In what follows, let $\BB(\RR;Y)$, $\BUC^\alpha(\RR;Y)$ etc.~denote the space of bounded $Y$-valued functions, the space of $\alpha$-H\"older continuous $Y$-valued functions etc., where $Y$ is some Banach space. Let $X$ be a fixed Banach space, and let $\ee^{t A}$ be a (contractive) analytic semigroup with generator $(A,\dom(A))$ thereon. The fractional powers of $-A$ are denoted by $\big((-A)^{\alpha},\dom((-A)^\alpha)\big)$. Denote by $F_\alpha$ the abstract Favard spaces for $X$ and $(\ee^{tA})_{t\geq 0}$, i.e.,
\begin{equation*}
F_\alpha:=\Bigl\{x\in X:
\|x\|_\alpha:=\|x\|+\sup_{t>0}\bigl\|\tfrac{\ee^{tA}x-x}{t^\alpha}\bigr\|<+\infty\Bigr\},
\end{equation*}
which becomes a Banach space if endowed with the norm
$\|\cdot\|_\alpha$. For every $\alpha,\beta\in (0,1)$ with
$\alpha>\beta$ we have continuous embeddings (see Engel and Nagel
\cite[Sec.~II.5.]{Engel-Nagel}):
\begin{equation*}
F_\alpha\hookrightarrow \dom((-A)^\beta)\hookrightarrow F_\beta.
\end{equation*}

\medskip\noindent Consider now the Banach space $\calX:=\BUC(\RR;X)$ and the semigroup
\begin{equation*}
(T(t)f)(s):=\ee^{tA}f(s-t)
\end{equation*}
thereon. We are interested in the Favard spaces $\calX_\alpha$ of
this semigroup.

\begin{proposition}\label{prop:favi}
In the above setting we have the following continuous inclusions:
\begin{equation*}
\BUC(\RR;\dom((-A)^\alpha))\cap \calX_\alpha \hookrightarrow
\BUC(\RR;\dom((-A)^\beta))\cap \BUC^\beta(\RR;X),
\end {equation*}
{for all
$0<\beta \leq \alpha<1$, and }
\begin{equation*}
\BUC^\alpha(\RR;X)\cap
\BUC(\RR;\dom((-A)^\alpha)) \hookrightarrow
\BUC(\RR;\dom((-A)^\beta))\cap \calX_\beta,
\end{equation*}
for all $0<\beta \leq \alpha<1$.
\end{proposition}
\begin{proof} We show the statement for $\beta=\alpha$, the rest then immediately follows.
We start with the second inclusion. For $f\in \BUC(\RR;X)$ we can
write
\begin{align*}
\sup_{t>0}\Bigl\|\frac{{T}(t)f-f}{t^\alpha}\Bigr\|&=\sup_{t>0}\sup_{s\in\RR}\frac{\|\ee^{tA} f(s-t)-f(s)\|}{t^\alpha}\\
&=\sup_{t>0}\sup_{s\in\RR}\frac{\|\ee^{tA} f(s)-f(s)+\ee^{tA} (f(s-t)-f(s))\|}{t^\alpha}\\
& \leq \sup_{s\in \RR}\|f(s)\|_{F_\alpha}+\|f\|_{\BUC^\alpha}.
\end{align*}

This shows that if $f\in \BB(\RR;F_\alpha)\cap \BUC^\alpha(\RR;X)$,
then $f\in \calX_\alpha$, and the inclusion is continuous, i.e.
\begin{equation*}
\|f\|_{\calX_\alpha}\leq c
\Bigl(\|f\|_{\BB(\RR;F_\alpha)}+\|f\|_{\BUC^\alpha(\RR;X)}\Bigr).
\end{equation*}

\medskip\noindent To see the first inclusion we use now that $A$ generates an analytic semigroup. If $f\in \BUC(\RR;\dom((-A)^{\alpha}))$,
then
\begin{align*}
&\sup_{t>0}\frac{\|\ee^{tA}f(s-t)-f(s-t)\|}{t^\alpha}=\sup_{t>0}\frac{\|(\ee^{tA}-I)(-A)^{-\alpha}(-A)^\alpha f(s-t)\|}{t^\alpha}\\
&\quad\leq C\sup_{t\in \RR} \|(-A)^\alpha f(s-t)\|\leq
C\|f\|_{\BUC(\RR;\dom((-A)^{\alpha}))}.
\end{align*}

This implies then
\begin{align*}
\sup_{t>0}\frac{\|f(s-t)-f(s)\|}{t^\alpha}\leq \sup_{s\in
\RR}\sup_{t>0}\Bigl\|\frac{{T}(t)f-f}{t^\alpha}\Bigr\|+C\|f\|_{\BUC(\RR;\dom((-A)^{\alpha}))}.
\end{align*}
The proof is complete.
\end{proof}

Now we are in the position to check the required commutator
condition and thus to prove Theorem \ref{thm:SchStrangsgrp}.

\begin{proof}[Proof of Theorem \ref{thm:SchStrangsgrp}]
Consider now the evolution semigroup corresponding to the
non-autonomous equation \eqref{eq:schroedinger}. The corresponding
generator is given formally as
\begin{equation*}
-\tfrac{\dd}{\dd s}+\Delta+V(t).
\end{equation*}

Take now $f\in \BUC^1(\RR;\Sob^2(\RR^d))$, and notice that then $f$
belongs to the domain $\dom(\mathcal{A})$. We calculate the
commutator of $\calA$ and $\calB$. We have
\begin{align*}
[\mathcal{A},\mathcal{B}]f=-V'(t)f(t)+(\Delta V(t))f+2\nabla
V(t)\cdot \nabla f(t).
\end{align*}
Now, if we assume that $V\in \BUC^1(\RR;\Ell^{\infty}(\RR^d))$ and
$V\in \BUC(\RR;W^{2,\infty}(\RR^d))$, then the first two terms can
be estimated by $c\|f\|$, so we have only to deal with the term
$2\nabla V\cdot \nabla f$, for which it suffices to estimate
$\partial _i f(t)$ for $i=1,\dots, d$. We have
\begin{align*}
\|\partial_i f(t)\|_2\leq
c\|\Delta^{1/2}f(t)\|_2\quad\mbox{($\partial_i$ is
$\Delta^{1/2}$-bounded on $\Ell^2$).}
\end{align*}
By Proposition \ref{prop:favi} this completes the proof of the
commutator condition \eqref{eq:commutator_cond} in the form
\begin{equation*}
\|[\calA,\calB]f\|\leq \|(-\calA)^\alpha f\|\quad\mbox{for all $f\in
\dom(\calA)$ with some given $\alpha \geq1/2$}.
\end{equation*}
Hence Theorem \ref{Thm:Jahnke-Lubich} yields the assertion.
\end{proof}


\section{Numerical examples illustrating the convergence}\label{sec:8}

In Section \ref{sec:7} we already introduced the non-autonomous
parabolic equation (sometimes also called imaginary time
Schr\"odinger equation)
\begin{equation}
\partial_t u(x,t) = \Delta u(x,t) + V(x,t)u(x,t)
\nonumber
\end{equation}
in $\RR^d$ with appropriate initial conditions with $V$ being a
smooth and bounded function. In the following we will apply the
sequential splitting introduced in Section \ref{sec:5} to the
sub-operators
\begin{align*}
A(t):=\Delta \quad \mbox{and} \quad B(t):=\mbox{multiplication by
${V(x,t)}$}.
\end{align*}
In Theorem \ref{th:firstproduct} we showed that the product formula
describing the sequential splitting is convergent also in the case
if we are able to solve the corresponding autonomous Cauchy problems
(Eq.~1)-(Eq.~2) with operators $A(r)$ and $B(r)$ for every time level
$r\in\RR$. We will use this result when constructing our numerical
scheme.

In order to illustrate numerically the convergence of the sequential
splitting and give an estimate on its order, let us consider the
following non-autonomous equation with boundary and initial
conditions:
\begin{equation}
\begin{cases}
\partial_t u(x,t) = \partial^2_x u(x,t) + V(x,t)u(x,t),\quad t\ge 0,\ x\in[0,1], \\
u(0,t) = u(1,t) = 0, \quad t\ge 0, \\
u(x,0) = u_0(x), \quad x\in[0,1]
\end{cases}
\label{eq:example}
\end{equation}
with functions $V(x,t)$ and $u_0(x)$ given later on in the example.

\subsection{Error analysis}

Let $(u_\mathrm{spl})_i^n$ denote the approximation of the exact
solution $u(i\Dx,n\Dt)$ of problem \eqref{eq:example} at time $n\Dt$
and at the grid point $i\Dx$ (with $n=0,...,N-1$ and $i=0,...,I-1$)
using sequential splitting. At this point the time-step
$\Dt=\tfrac{1}{N-1}$ and the grid size $\Dx=\tfrac{1}{I-1}$ have
certain given values. We call
$(u_\mathrm{spl})^n=\big((u_\mathrm{spl})_0^n,(u_\mathrm{spl})_1^n,...,(u_\mathrm{spl})_{I-1}^n\big)$,
$n=0,1,..,N-1$, the \emph{split solution} of problem
\eqref{eq:example}. As already seen, the order of the splitting
procedure can be estimated with the help of the splitting error
defined by
\begin{equation*}
\mathcal{E}_\mathrm{spl}^n:=\| u^n - u_\mathrm{spl}^n \|
\end{equation*}
where $u^n=(u_0^n,u_1^n,...,u_{I-1}^n)$ with $u_i^n=u(i\Dx,n\Dt)$,
$i=0,1,...,I-1$. With this notation the splitting procedure (or an
arbitrary finite difference method) is of order $p>0$ if for
sufficiently smooth initial values there is a constant $C>0$ such
that for all $t\in [0,t_0]$ we have
\begin{align*}
\mathcal{E}_\mathrm{spl}^n &\leq \tfrac{C}{n^p}, \intertext{or, if
the method is stable, equivalently,} \mathcal{E}_\mathrm{spl}^1
&\leq C'\Dt^{p+1}.
\end{align*}
 In general, the exact solution of problem \eqref{eq:example} is unknown, therefore, the local splitting error $\mathcal{E}_\mathrm{spl}^1$ is to be estimated as well. To this end we compute a so-called \emph{reference solution} $u_\mathrm{ref}^n$ on a finer space grid using no splitting procedure. Then the order $p$ of the splitting procedure can be determined as follows. From the definition of $p$ we have $\mathcal{E}_\mathrm{spl}^1\le C\Dt^{p+1}$. Approximating $u^n$ with $u_\mathrm{ref}^n$, we obtain $\mathcal{E}_\mathrm{spl}^1\approx\widetilde{\mathcal{E}}_\mathrm{spl}^1:=\|u_\mathrm{ref}^1-u_\mathrm{spl}^1\|\le C\Dt^{p+1}$. Thus,
\begin{equation*}
\log\widetilde{\mathcal{E}}_\mathrm{spl}^1 \le (p+1)\log\Dt + \log
C.
\end{equation*}
Then we can estimate $p$ by computing the approximate local
splitting error $\widetilde{\mathcal{E}}_\mathrm{spl}^1$ for many
different values of the time-step $\Dt$, plotting the logarithm of
the results, and fitting a line of form $y(w)=aw+b$ to them. Hence,
$a\approx p+1$ and $b\approx\log C$. Note, however, that the split
solution contains not only the splitting error but also a certain
amount of error originating from the spatial and temporal
discretization. In what follows we show how to determine the
numerical solutions $u_\mathrm{ref}^1$ and $u_\mathrm{spl}^1$.

\medskip
We also note that it is reasonable to compute a relative local error
defined as
\begin{equation*}
\mathcal{E}_\mathrm{loc}=\frac{\widetilde{\mathcal{E}}_\mathrm{spl}^1}{\|u_\mathrm{ref}^1\|}
\end{equation*}
because this yields the ratio how the split solution differs from
the reference solution.

\subsection{Numerical scheme}

In order to solve numerically the problem \eqref{eq:example} we
should discretize it in both space and time. For the temporal
discretization we used the Crank-Nicholson method, and we chose the
finite difference method for the spatial discretization.

\subsubsection{Reference solution}

As mentioned above, we need a reference solution $u_\mathrm{ref}^n$
computed without using splitting procedures. After discretizing the
equation, we obtain the following numerical scheme for determining
$(u_\mathrm{ref})_i^{n+1}$:
\begin{equation}
(u_\mathrm{ref}^{n+1})_i \\
=\big(1-(H_\mathrm{ref})_i^{n+1}\big)^{-1}\big(1+(H_\mathrm{ref})_i^n\big)(u_\mathrm{ref}^{n})_i
\label{eq:scheme}
\end{equation}
with
\begin{equation*}
(H_\mathrm{ref})_i^n=\frac{\Dt}{2}\left(\frac{u_{i+1}^{n+1}-2u_i^{n+1}+u_{i-1}^{n+1}}{\Dx^2}+V_i^n\right),
\end{equation*}
where $V_i^n:=V(i\Dx,n\Dt)$.

\subsubsection{Split solution}

Application of sequential splitting means that instead of the whole
problem \eqref{eq:example} two sub-problems are solved. In our
examples the first sub-problem corresponds to the diffusion equation
$\partial_tu_A(x,t)=\partial_x^2u_A(x,t)$. Its numerical solution
$u_A^n$ can also be computed using Crank-Nicholson temporal and
finite difference spatial discretization methods. Then we obtain the
following numerical scheme similar to \eqref{eq:scheme}:
\begin{equation}
(u_A^{n+1})_i \\
=\big(1-(H_A)_i^{n+1}\big)^{-1}\big(1+(H_A)_i^n\big)(u_A^{n})_i
\label{eq:scheme_sq1}
\end{equation}
with
\begin{equation*}
(H_A)_i^n=\frac{\Dt}{2}\frac{u_{i+1}^{n+1}-2u_i^{n+1}+u_{i-1}^{n+1}}{\Dx^2}.
\end{equation*}
The second sub-problem has the multiplication operator by ${V(x,t)}$
on its right-hand side, i.e. $\partial_tu_B(x,t)=V(x,t)u_B(x,t)$. We
refer again to Theorem \ref{th:firstproduct} and take the function
$V$ only at time levels $t=n\Dt$, $n=0,1,...,N-1$. In this
(autonomous) case the exact solution
$u_B(x,t)=\ee^{t{V(x,n\Dt)}}u_0(x)$ is known. At the $n^\mathrm{th}$
time level and on the space grid it has the form
\begin{equation}
(u_B^n)_i=u_B(i\Dx,n\Dt)=\ee^{\Dt V(i\Dx,n\Dt)}u_0(i\Dx).
\label{eq:scheme_sq2}
\end{equation}
Due to the product formula \eqref{spl_sq_1}, the split solution
$u_\mathrm{spl}^n $ is given by the following algorithm:
\begin{align*}
& \mathbf{for} \quad i=0,...,I-1 \\
& \qquad \mbox{initial function:} \quad (u_A)_i^0:=u_0(i\Dx) \\
& \mathbf{end} \\
& \mathbf{for} \quad n=0,1,...,N-1 \\
& \qquad \mathbf{for} \quad i=0,1,...,I-1 \\
& \qquad \qquad \mbox{solve the first sub-problem using \eqref{eq:scheme_sq1}} \quad \Longrightarrow \quad (u_A)_i^n \\
& \qquad \mathbf{end} \\
& \qquad \mathbf{for} \quad i=0,1,...,I-1 \\
& \qquad \qquad \mbox{solve the second sub-problem using \eqref{eq:scheme_sq2}} \quad \Longrightarrow \quad (u_B)_i^n \\
& \qquad \mathbf{end} \\
& \mathbf{end} \\
& \mbox{split solution:} \quad u_\mathrm{spl}^{N-1}:=u_B^{N-1}
\end{align*}

\subsection{Numerical results}
Now we present some numerical results on the following example.



Choose
\begin{equation*}
V(x,t)=t-500x^2 \quad \mbox{and} \quad u_0(x)=\ee^{-50(x-0.4)^2}.
\end{equation*}
Since the exact solution is unknown in this case, we should estimate the local splitting error using the reference solution instead of the exact one. Then the relative local splitting error $\mathcal{E}_\mathrm{loc}$ and its order $p$ can be measured.
\begin{figure}[!ht]
\begin{center}
\includegraphics[width=8cm]{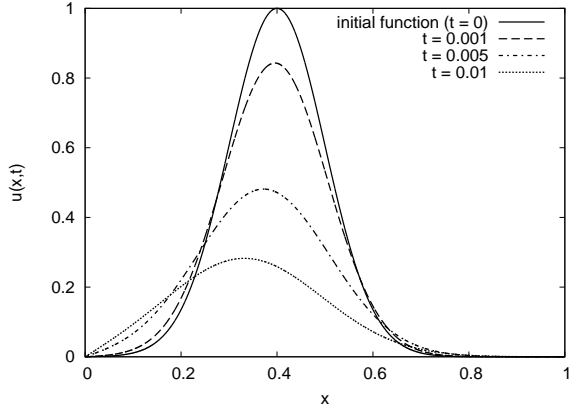}
\begin{parbox}{12cm}
{\caption{\label{fig:pupok}\footnotesize Numerical solution of
equation \eqref{eq:example} at time levels $t=0$, $t=10^{-3}$,
$t=5\cdot10^{-3}$, and $t=10^{-2}$, respectively.}}
\end{parbox}
\end{center}
\end{figure}
\begin{figure}[!ht]
\begin{center}
\includegraphics[width=8cm]{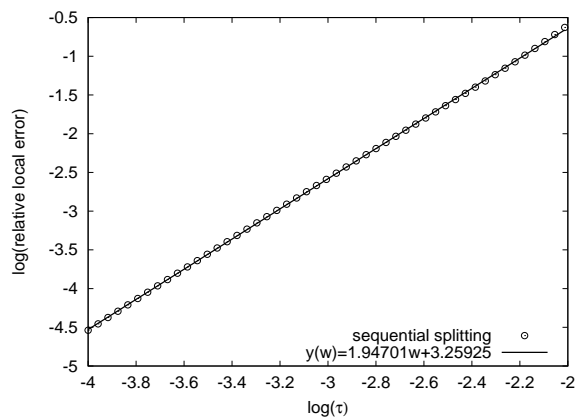}
\begin{parbox}{12cm}
{\caption{\label{fig:egyenes}\footnotesize Results obtained by
applying the sequential splitting with various time steps (dots),
and the line $y(w)=aw+b$ fitted to them with parameters
$a=1.9470\approx p+1$ and $b=3.25925$.}}
\end{parbox}
\end{center}
\end{figure}

On Figure \ref{fig:pupok} the time-behavior of the reference
solution can be seen at the four time levels $t=0$, $t=10^{-3}$,
$t=5\cdot10^{-3}$, and $t=10^{-2}$, respectively. The effect of the
diffusion can be clearly observed. Figure \ref{fig:egyenes} shows
the result of the fitting. The dots correspond to
$\log(\mathcal{E}_\mathrm{loc})$ for the various step sizes. The
line fitted to these points has the form $y(\log(\Dt))=a\log(\Dt)+b$
with $a=1.9470$ and $b=3.25925$. As mentioned above, the order of
the splitting procedure $p$ can be estimated by $a-1\approx 1$, that
is, the sequential splitting is of first order.
%


\section*{Acknowledgments}
A.~B\'atkai was supported by the Alexander von Humboldt-Stiftung. We
thank Wolfgang Arendt (Ulm), Roland Schnaubelt (Karlsruhe) and Alexander Ostermann (Innsbruck) for interesting and useful
discussions. The European Union and the European Social Fund have provided financial
support to the project under the grant agreement no. T\'{A}MOP-4.2.1/B-09/1/KMR-2010-0003.


\parindent0pt

\end{document}